\documentclass[11pt]{amsart}

\usepackage[utf8x]{inputenc} 	      
\usepackage{lmodern}                
\usepackage[T1]{fontenc}            
\usepackage[dvips]{graphicx}
\usepackage{xcolor}
\usepackage{pstricks-add}
\usepackage{latexsym, stmaryrd}
\usepackage{amsmath,amsthm,amssymb}

\newtheorem{theorem}{Theorem}[section]
\newtheorem{lemma}[theorem]{Lemma}

\newtheorem{proposition}[theorem]{Proposition}

\newtheorem{remark}[theorem]{Remark}

\numberwithin{equation}{section}

\newcommand{\ens}[1]{\mathbb{#1}}

\newcommand{\N}{\mathbb{N}}
\newcommand{\NN}{\mathbb{N}}
\newcommand{\Z}{\mathbb{Z}}
\newcommand{\R}{\mathbb{R}}

\def\cal{\mathcal}

\def\derpar#1#2{\frac{\partial#1}{\partial#2}}

\def\ve{\varepsilon}

\def\signcm{\bigskip\bigskip\hspace{80mm}
\vbox{{\sc C. Mouhot\par\vspace{3mm}
DPMMS, Centre for Mathematical Sciences\par
University of Cambridge\par
Wilberforce Road\par
Cambridge CB3 0WA\par
UNITED KINGDOM\par\vspace{3mm}
e-mail:} C.Mouhot@dpmms.cam.ac.uk }}

\def\signlp{\bigskip\bigskip\hspace{80mm}
\vbox{{\sc L. Pareschi\par\vspace{3mm} DMI, Universit\`{a} di
Ferrara\par Via Machiavelli 35\par I-44121 Ferrara \par ITALY
\par\vspace{3mm} e-mail:} lorenzo.pareschi@unife.it }}

\def\signtr{\bigskip\bigskip\hspace{80mm}
\vbox{{\sc T. Rey\par\vspace{3mm}
 CSCAMM, University of Maryland\par
 CSIC Building, Paint Branch Drive\par
 College Park, MD 20740 \par 
 USA\par
 \vspace{3mm} e-mail:} trey@cscamm.umd.edu }}

\begin{document}

\title[Fast summation methods for the Boltzmann equation] {Convolutive
  decomposition and fast summation methods for discrete-velocity
  approximations of the Boltzmann equation}

\author{Cl\'ement Mouhot, Lorenzo Pareschi and Thomas Rey}

\hyphenation{bounda-ry rea-so-na-ble be-ha-vior pro-per-ties
cha-rac-te-ris-tic}

\begin{abstract}
  Discrete-velocity approximations represent a popular way for
  computing the Boltzmann collision operator. The direct numerical
  evaluation of such methods involve a prohibitive cost,
  typically $O(N^{2d+1})$ where $d$ is the dimension
  of the velocity space. In this paper, following the ideas introduced
  in~\cite{MoPa:2006b, MoPa:2006}, we derive fast summation techniques
  for the evaluation of discrete-velocity schemes which permits to
  reduce the computational cost from $O(N^{2d+1})$
  to $O(\bar{N}^d N^d\log_2 N)$, $\bar{N} \ll N$, with
  almost no loss of accuracy.
\end{abstract}

\maketitle



\tableofcontents

\section{Introduction}
\label{sec:intro}

Among deterministic methods to approximate the Boltzmann collision
integral, one of the most popular is represented by {discrete
velocity models} (DVM). These methods~\cite{Bu:96, MaSc:FBE:92,
BoVaPaSc:cons:95, CoRoSc:hom:92, HePa:DVM:02, RoSc:quad:94,
MiSc:2000, BoVi:2008} are based on a regular grid in the velocity
field and construct a discrete collision mechanics on the points
of the quadrature rule in order to preserve the main physical
properties.

As compared to Monte-Carlo methods, these methods have certain
number of assets: accuracy, absence of statistical fluctuations,
and the fact that the distribution function is explicitly
represented in the velocity space. However their computational
cost is more than quadratic and they cannot compete with the
linear cost of a Monte Carlo approach. Indeed the ``naive'' cost
of a product quadrature formula for the $(d-1)+d$ fold Boltzmann
collision integral in dimension $d$ is $O(M^{d-1}N^{d})$, where
$M$ is related to the angle and $N$ to the velocity
discretizations. More concretely Buet presented in \cite{Bu:96} a
DVM algorithm widely used since then in $O(N^{2d+1+\ve})$ for all
$\ve > 0$ (and a constant depending on $\ve$); Michel and
Schneider algorithm in \cite{MiSc:2000} is $O(N^{2d
  +\delta})$ where $\delta$ depends on $d$ and is close to $1$;
finally the method of Panferov and Heinz \cite{HePa:DVM:02} is
$O(N^{2d+1})$. For this reason several acceleration techniques for DVM
have been proposed in the past literature. We do not seek to review
them here, and refer the reader to~\cite{Bu:96, KoPaRuWa:cons:08,
  PaMa:05, PlWa:cons:00, VaNa:cons:03}.



More recently a new class of methods based on the use of spectral
techniques in the velocity space has attracted the attention of the
scientific community. The method first developed for the Boltzmann
equation in~\cite{PePa:96} is based on a Fourier-Galerkin
approximation of the integral collision operator.  As shown
in~\cite{PaRu:spec:00, PaRu:stab:00} the method permits to obtain
spectrally accurate solution at a reduced computational cost of
$O(N^{2d})$. A  proof of stability and convergence  for
this method has been given in~\cite{FiMo:2011}. Finally the method has
been extended to the case of the quantum Boltzmann collision
operator~\cite{FiHuJi:qbe:11, HuYi:qbe:11}. Other methods based on
spectral techniques have been developed in~\cite{BoRj:HS:99,
  GaTh:spe:09}.

One of the major differences between DVM and spectral methods is
that in the latter the interaction kernel of the Boltzmann
collision integral is not modified in order to obtain a
conservative equation on a bounded domain. This aspect has a
profound influence on the resulting structure of the algorithm
since most of the symmetries which are present in the original
operator are preserved. Using this fact, in~\cite{MoPa:2006b,
MoPa:2006}, the authors developed a numerical technique based on
the Fast Fourier Transform (FFT) that permits to reduce the cost
of spectral method from $O(N^{2d})$ to $O(M^{d-1}N^d\log_2 N)$
where $M$ is the number of angle discretizations. These ideas have
been successfully used in~\cite{FiMoPa:2006} to compute space non
homogeneous solutions of the Boltzmann equation.


%

%

In this paper we will consider general discrete velocity
approximation of the Boltzmann equation without any modification
to the original collision kernel and show how the FFT techniques
developed in~\cite{MoPa:2006b, MoPa:2006} can be adapted to this
case to obtain acceleration algorithms. In this way, for a
particular class interactions using a Carleman-like representation
of the collision operator we are able to derive discrete velocity
approximations that can be evaluated through fast algorithms at a
cost of $O(\bar{N}^dN^d \log_2 N)$, $\bar{N} \ll N$. The class of
interactions includes {\em Maxwellian molecules} in dimension two
and {\em hard spheres} molecules in dimension three.

Let us emphasize here that a detailed analysis of the
computational complexity in DVM is non trivial since imposing
conservations on the points of the quadrature rule originates a
summation formula that requires the exact enumeration of the set
of involved orthogonal directions in $\Z^d$.


The rest of the paper is organized in the following way. In the next
Section we introduce briefly the Boltzmann equation and give a
Carleman-like representation of the collision operator which is used
as a starting point for the development of our methods. In
Section~\ref{sec:fastDVM} a fast DVM method is
  introduced together with  a detailed analysis of its computational
complexity. In Section~\ref{sec:Numerics}, we present
some numerical results obtained with the fast and the classical DVM
methods.

\section{Preliminaries}\label{sec:rep}

\subsection{The Boltzmann equation}

The Boltzmann equation describes the behavior of a dilute gas of
particles when the only interactions taken into account are binary
elastic collisions. It reads for $x,v \in \R^d$ ($d \ge 2$)
 \begin{equation*}
 \derpar{f}{t} + v \cdot \nabla_x f = Q(f,f)
 \end{equation*}
where $f(t,x,v)$ is the time-dependent particle distribution
function in the phase space. The Boltzmann collision operator $Q$
is a quadratic operator local in $(t,x)$. The time and position
acts only as parameters in $Q$ and therefore will be omitted in
its description
 \begin{equation}\label{eq:Q}
 Q (f,f)(v) = \int_{v_* \in \R^d}
 \int_{\sigma \in \ens{S}^{d-1}}  B(\cos \theta,|v-v_*|) \,
 \left[ f'_* f' - f_* f \right] \, d\sigma \, dv_*.
 \end{equation}
In \eqref{eq:Q} we used the shorthand $f = f(v)$, $f_* = f(v_*)$,
$f ^{'} = f(v')$, $f_* ^{'} = f(v_* ^{'})$. The velocities of the
colliding pairs $(v,v_*)$ and $(v',v'_*)$ are related by
 \begin{equation*}
 v' = \frac{v+v_*}{2} + \frac{|v-v_*|}{2} \sigma, \qquad
 v'_* = \frac{v+v^*}{2} - \frac{|v-v_*|}{2} \sigma\nonumber.
 \end{equation*}
The collision kernel $B$ is a non-negative function which by
physical arguments of invariance only depends on $|v-v_*|$ and
$\cos \theta = {\widehat g} \cdot \sigma$ (where ${\widehat g} =
(v-v_*)/|v-v_*|$).

Boltzmann's collision operator has the fundamental properties of
conserving mass, momentum and energy
 \begin{equation*}
 \int_{v\in{\R}^d}Q(f,f)\phi(v)\,dv = 0, \qquad
 \phi(v)=1,v,|v|^2 
 \end{equation*}
and satisfies the well-known Boltzmann's $H$-theorem
 \begin{equation*} 
 - \frac{d}{dt} \int_{v\in{\R}^d} f \log f \, dv = - \int_{v\in{\R}^d} Q(f,f)\log(f) \, dv \geq 0.
 \end{equation*}
The functional $- \int f \log f$ is the entropy of the solution.
Boltzmann $H$-theorem implies that any equilibrium distribution
function, i.e. any function which is a maximum of the entropy,
has the form of a locally Maxwellian distribution
 \begin{equation*}
 M(\rho,u,T)(v)=\frac{\rho}{(2\pi T)^{d/2}}
 \exp \left( - \frac{\vert u - v \vert^2} {2T} \right), 
 \end{equation*}
where $\rho,\,u,\,T$ are the density, mean velocity and
temperature of the gas
 \begin{equation*}
 \rho = \int_{v\in{\R}^d}f(v)dv, \quad u =
 \frac{1}{\rho}\int_{v\in{\R}^d}vf(v)dv, \quad T = \frac{1}{d\rho}
 \int_{v\in{\R}^d}\vert u - v \vert^2f(v)dv.
 \end{equation*}
For further details on the physical background and derivation of
the Boltzmann equation we refer to~\cite{CIP:94}
and~\cite{Vill:hand}.

\subsection{Carleman-like representation in bounded domains}

In this short paragraph we shall approximate the collision
operator on a bounded domain starting from a representation which
somehow conserves more symmetries of the collision operator when
one truncates it in a bounded domain. This representation was used
in~\cite{BoRj:HS:97, BoRj:HS:99, BoRj:maxw:98, IbRj:02, MoPa:2006}
and is close to the classical Carleman representation
(cf.~\cite{Carl:EB:32}).

The starting point of this representation is the identity
 \begin{equation}\label{eq:dirac}
 \frac{1}{2} \, \int_{\ens{S}^{d-1}} F(|u|\sigma - u) \, d\sigma
 = \frac{1}{|u|^{d-2}} \, \int_{\R^d} \delta(2 \, x \cdot u + |x|^2) \, F(x) \, dx.
 \end{equation}
It can be verified easily by completing the square in the delta Dirac function, taking the spherical
coordinate $x=r \, \sigma$ and performing the change of variable $r^2 = s$. Then, setting $u = v-v_*$
and $r = |u|$, we have the following Lemma.

\begin{lemma}[Cf. \cite{MoPa:2006}, subsection~2.1]
Introducing the change of variables
\[
x=\frac12 \, r\, \sigma,\quad y=v_*-v-x,
\]
the collision operator (\ref{eq:Q}) can be rewritten in the form
 \begin{multline*}
 Q (f,f)(v) = \int_{x \in \R^d} \int_{y \in \R^d}
 \tilde{B}(x,y)\,
 \delta(x \cdot y)
 \left[ f(v + y) \, f(v+ x) - f(v+x+y) \, f(v) \right] \, dx \,
 dy,
 \end{multline*}
 where
 \begin{equation}\label{eq:Btilde}
 \tilde{B}(x,y) = \tilde{B}(|x|,|y|) =
 2^{d-1} \, B\left( \frac{|x|}{\sqrt{|x|^2+|y|^2}},\sqrt{|x|^2+|y|^2} \right) \, (|x|^2+|y|^2)^{-\frac{d-2}2}.
 \end{equation}

\end{lemma}

Now let us consider the bounded domain $\mathcal{D}_T=[-T,T]^d$
($0< T < +\infty$).
First one can remove the collisions connecting with some points
out of the box. This is the natural preliminary stage for deriving
conservative schemes based on the discretization of the velocity.
In this case there is no need for a truncation on the modulus of
$x$ and $y$ since we impose them to stay in the box.
It yields
 \begin{multline*}
 Q^{\mbox{tr}} (f,f)(v) = \int \int_{\big\{
 x, \, y \, \in \, \R^d \ | \ v+x, \, v+y, \, v+x+y \, \in \, \mathcal{D}_T \big\} }
 \tilde{B}(x,y) \,
 \delta(x \cdot y) \\
 \left[ f(v + y) \, f(v+ x) - f(v+x+y) \, f(v) \right] \, dx \, dy
 \end{multline*}
defined for $v \in  \mathcal{D}_T$. One can easily check that the
following weak form is satisfied by this operator
 \begin{multline}\label{eq:weaktr}
 \int Q^{\mbox{tr}} (f,f) \, \varphi(v) \, dv = \frac{1}{4} \,
 \int \int \int_{\big\{ v,\, x, \, y \, \in \, \R^d \ | \
 v,\, v+x,\, v+y,\, v+x+y \, \in \, \mathcal{D}_T \big\}}
 \tilde{B}(x,y) \, \delta(x \cdot y) \\
 f(v+x+y) \, f(v) \left[ \varphi(v + y) + \varphi(v+ x) - \varphi(v+x+y) - \varphi(v) \right] \, dv \, dx \, dy
 \end{multline}
and this implies conservation of mass, momentum and energy as well
as the $H$-theorem on the entropy. The problem of this truncation
on a bounded domain is the fact that we have changed the collision
kernel itself by adding some artificial dependence on $v,
v_*,v',v'_*$. In this way convolution-like properties are broken.

A different approach consists in truncating the integration in $x$
and $y$ by setting them to vary in $\mathcal{B}_R$, the ball of
center $0$ and radius $R$. For a compactly supported function $f$
with support $\mathcal{B}_S$, we take $R=S$ in order to obtain all
possible collisions.
Since we aim at using the FFT algorithm to evaluate the resulting
quadrature approximation, and hence we will make use of periodic
distribution functions, we must take into account the aliasing
effect due to periods superposition in the Fourier space. As for
the spectral method a geometrical argument
(see~\cite{PaRu:spec:00} for further details)
shows that using the periodicity of the function it is enough to
take $T \ge (3 + \sqrt{2})S/2$ to prevent intersections of the
regions where $f$ is different from zero.

The operator now reads
 \begin{multline}\label{eq:repR}
 Q^R (f,f)(v) = \int_{x \in \mathcal{B}_R} \int_{y \in \mathcal{B}_R}
 \tilde{B}(x,y) \, \delta(x \cdot y) \\
 \left[ f(v + y) f(v+ x) - f(v+x+y) f(v) \right] \, dx \, dy
 \end{multline}
 for $v \in \mathcal{D}_T$. The interest of this representation is to
 preserve the real collision kernel and its
 properties. It is easy to check that, except for the
 aliasing effect, the operator preserves all the original
 conservation properties, see the weak form in
 equation~\eqref{eq:weakpe}.

In order to understand the possible effect of periods
superposition we can rely on the following weak form valid for
any function $\varphi$ {\em periodic} on $\mathcal{D}_T$
 \begin{multline}\label{eq:weakpe}
 \int_{\mathcal{D}_T} Q^R (f,f) \, \varphi(v) \, dv = \frac{1}{4} \,
 \int_{v \in \mathcal{D}_T} \int_{x \in \mathcal{B}_R} \int_{y \in \mathcal{B}_R}
 \tilde{B}(x,y) \, \delta(x \cdot y) \\
 f(v+x+y) f(v) \left[ \varphi(v + y) + \varphi(v+ x) - \varphi(v+x+y) - \varphi(v) \right] \, dv \, dx \,
 dy.
 \end{multline}

About the conservation properties one can show that
 \begin{enumerate}
 \item The only invariant $\varphi$ is $1$: it is the only periodic function on $\mathcal{D}_T$
 such that
  \[ \varphi(v + y) + \varphi(v+ x) - \varphi(v+x+y) - \varphi(v) = 0 \]
 for any $v \in \mathcal{D}_T$ and $x \bot y \in \mathcal{B}_R$ (see~\cite{Cerc:75} for instance).
 It means that the mass is locally conserved but not necessarily the momentum and energy.
 \item When $f$ is even there is {\em global} conservation of momentum, which is $0$ in this case.
 Indeed $Q^R$ preserves the parity property of the solution, which can be checked using
 the change of variable $x \to -x$, $y \to -y$.
 \item The collision operator satisfies formally the $H$-theorem
  \[ \int_{v\in{\R}^d} Q^R (f,f)\log(f) \, dv \leq 0. \]
 \item If $f$ has compact support included in $\mathcal{B}_S$ with $T \ge (3 + \sqrt{2})S/2$
 (no-aliasing condition, see \cite{PaRu:spec:00} for a detailed discussion) and $R =  S$,
 then no unphysical collisions occur
 and thus mass, momentum and energy are preserved. Obviously this compactness is not preserved with
 time since the collision operator spreads the support of $f$ by a factor $\sqrt{2}$.
 \end{enumerate}

\subsection{Application to discrete-velocity models}

The representation $Q^R$ of this section can also be used to
derive discrete velocity models (DVM). Any DVM can be written in
the general form
 \begin{equation}\label{eq:TypicalDVM}
 D_i (f,f)= \sum_{j,k,l \, \in \, \Z^d} \Gamma_{i,j} ^{k,l} \big[ f_k f_l - f_i f_j
 \big],
 \end{equation}
 where $D_i$ denotes the discrete Boltzmann collision operator and the
 integer indexes refer to the points in the computational grid.

In order to keep conservations the coefficients
$\Gamma_{i,j}^{k,l}$ are defined by
 \begin{equation}\label{eq:Gammaijkl}
   \Gamma_{i,j} ^{k,l} = {\bf 1}(i+j-k-l) \, {\bf 1}(|i|^2 + |j|^2 - |k|^2 - |l|^2 ) \,
   B(|k-i|,|l-j|) \, w_{i,j}^{k,l}
 \end{equation}
 where ${\bf 1}$ denotes the function on $\Z$ defined by ${\bf
   1}(z)=1$ if $z=0$ and $0$ elsewhere, and $w_{i,j}^{k,l}>0$ are the
 weights of the quadrature formula, which characterize the different
 DVM.  The function $B >0$ is the discrete collision kernel. One can
 check on this formulation that the scheme satisfies the usual
 conservation laws and entropy inequality (see~\cite{PaIl:88, CaGa:80}
 and the references therein).  More details on the DVM schemes can
 also be found in \cite{CaGa:80}.

Thanks to equations \eqref{eq:TypicalDVM} and \eqref{eq:Gammaijkl}, we
can write at the discrete level the same representation as in
the continuous case
 \begin{equation*}
   D_i (f,f)= \sum_{k,l \, \in \, \Z^d} \widetilde{\Gamma}_{k,l} \big[
   f_{i+k}
   f_{i+l} - f_i f_{i+k+l} \big]
 \end{equation*}
with
 \begin{equation*}
   \widetilde{\Gamma}_{k,l} = 2^{d-1} \, B\left(
     \frac{|k|}{\sqrt{|k|^2+|l|^2}},\sqrt{|k|^2+|l|^2} \right) \,
   (|k|^2+|l|^2)^{-\frac{d-2}2} {\bf 1} (k \cdot l) \, w_{k,l}.
 \end{equation*}
This is coherent with the DVM obtained by quadrature starting from
the Carleman representation in~\cite{HePa:DVM:02}.

Now again when one is interested to compute the DVM in a bounded
domain there are two possibilities. First as in the case of
$Q^{\mbox{tr}}$ one can force the discrete velocities to stay in a box,
which yields for $i\in \llbracket -N,N\rrbracket^d$ (again using the one index
notation for $d$-dimensional sums)
 \begin{equation*}
 D^{\mbox{tr}} _i(f,f) = \sum_{\substack{k,l \\-N\leq \, i+k,\, i+l,\, i+k+l \leq N}}
 \widetilde{\Gamma}_{k,l} \big[ f_{i+k} f_{i+l} - f_i f_{i+k+l} \big].
 \end{equation*}
 This new discrete operator is completely conservative but the
 collision kernel is not invariant anymore according to $i$, which
 breaks the convolution properties and then prevents the derivation of
 a fast algorithm.

The other possibility is to periodize the function $f$ over the
box and truncate the sum in $k$ and $l$. It yields for a given
truncation parameter $\tilde{N} \in \N^*$
 \begin{equation}\label{eq:DR}
 D^{\tilde{N}} _i(f,f) = \sum_{-\tilde{N} \leq k,l \leq \tilde{N}}
 \widetilde{\Gamma}_{k,l} \big[ f_{i+k} f_{i+l} - f_i f_{i+k+l} \big],
 \end{equation}
for any $i\in \llbracket -N,N\rrbracket^d$.
%

It is easy to see that $D^{\tilde{N}}$ satisfies exactly a discrete
weak form and conservation properties similar to $Q^{R}$. Let us
briefly state and sketch the proof of the conservation and stability
properties of the scheme.
\begin{proposition}
  \label{prop:stab}
  Assume that the quadrature weight $w_{i,j}^{k,l}>0$ are
  positive. Consider some truncation numbers $\tilde N \le N \in \N^*$
  and some non-negative initial data $f_i(0) \ge 0$, $i \in \llbracket
  -N,N\rrbracket^d$. Then the discrete evolution equation
 \[
 \partial_t f_i = D^{\tilde N} _i (f,f) = \sum_{-\tilde{N} \leq k,l \leq \tilde{N}}
 \widetilde{\Gamma}_{k,l} \big[ f_{i+k} f_{i+l} - f_i f_{i+k+l} \big],
\qquad i \in \llbracket -N,N\rrbracket^d,
\]
is globally well-posed in $\R^{\llbracket -N,N\rrbracket^d}$. Moreover
the coefficients $f_i(t)$ are non-negative for all time, and
\[
\forall \, t \ge 0, \quad \sum_{i \in \llbracket -N,N\rrbracket^d}
f_i(t) = \sum_{i \in \llbracket -N,N\rrbracket^d} f_i(0).
\]
\end{proposition}

\begin{remark}
  The DVM scheme we consider therefore preserves non-negativity, but
  let us also emphasize that it preserves momentum and energy up to
  aliasing issues. This is different from spectral methods where the
  truncation of Fourier modes introduces an additional error in the
  conservation laws. Concerning the spectral method, stability and
  convergence have been proved recently in~\cite{FiMoPa:2006} to hold
  in $L^1$ but only asymptotically, i.e. for $N$ big enough related to
  the initial data.
\end{remark}

\begin{proof}[Proof of Proposition~\ref{prop:stab}]
  
  We have the following $L^1$-like estimate
	\begin{align}
	  \frac{d}{dt} \sum_{i \in \llbracket -N,N\rrbracket^d} \left| f_i(t) \right | & 
	    = \sum_{i \in \llbracket -N,N\rrbracket^d} \left | 
	      \sum_{-\tilde{N} \leq k,l \leq \tilde{N}} \widetilde{\Gamma}_{k,l} 
	      \big[ f_{i+k} f_{i+l} - f_i f_{i+k+l} \big] \right| \notag \\ 
	    & \leq C \, \left( \sum_{i \in \llbracket -N,N\rrbracket^d} |f_i| \right)^2. \label{ineq:L1est}
	\end{align}
	The use of a Gr\"onwall argument then gives the local well-posedness of the scheme in $\R^{\llbracket -N,N\rrbracket^d}$.
	Moreover, given a local solution $f_i(t)$, for $t \in [0, T]$ and $T > 0$, it is clear by construction that the conservation of mass holds.
	
	The proof of preservation of non-negativity for this solution is essentially contained in the pioneering work of Carleman \cite{Carl:EB:32}. We will sketch its proof in the following. Let us rewrite the system of ordinary differential equations satisfied by $f_i$ for a fixed $ i \in \llbracket -N, N \rrbracket^d $ as
	\begin{equation}
	  \label{eq:SysDiffEqDVM}
	  \frac{d}{dt} f_i + f_i \sum_{-\tilde{N} \leq k,l \leq \tilde{N}} f_{i+k+l}= \sum_{-\tilde{N} \leq k,l \leq \tilde{N}} \widetilde{\Gamma}_{k,l} \, f_{i+k}\, f_{i+l}.
	\end{equation}
	Let us assume by contradiction that we have
	\begin{equation*}
	   \left \{ \begin{aligned}
	     & f_j\,(t) > 0, \quad \forall t \in [0,T[, \quad \forall j \in \llbracket -N, N \rrbracket^d, \\
	     & ~ \\
	     & f_i\,(T) = 0.
	   \end{aligned} \right.
	\end{equation*}
	Then, we have necessarily
	\[ f_i'\,(T) \leq 0,\]
	and thus, according to \eqref{eq:SysDiffEqDVM}, 
	\[\sum_{-\tilde{N} \leq k,l \leq \tilde{N}} \widetilde{\Gamma}_{k,l} \, f_{i+k}(T) \, f_{i+l}(T) \leq 0. \] 
	By continuity in time of $f_j$, it comes that
	\[ f_j(T) = 0, \quad  \forall j \in \llbracket -N, N \rrbracket^d. \]
	As these conditions implies using \eqref{eq:SysDiffEqDVM} that $f_j(t) = 0$ for all $t \in [0,T]$, we have a contradiction with the non-negativity of the initial condition.

  Finally, the conservations of mass and non-negativity implies the preservation of $L^1$ norm, and we can iterate the argument giving the local well-posedness (still using inequality \eqref{ineq:L1est}) to obtain the global well-posedness of the scheme.
	
\end{proof}

Finally one can derive the following consistency result
from~\cite[Theorem~3]{HePa:DVM:02} in the case of hard spheres
collision kernel with $d=3$
 \begin{theorem}\label{theo:HePa}
   Assume that $f,g \in C^k (\R^3)$ ($k \ge 1$) with compact support
   $\mathcal{B}_S$. The uniform grid of step $h$ is constructed on the
   box $\,\cal{D}_T$ with the no-aliasing condition $T \ge (3 +
   \sqrt{2})S/2$. Then for $\tilde{N} = [S/h]$ (where $[\,\cdot\,]$
   denotes the floor function) and $h >0$ sufficiently small,
   \[ \left \| Q(g,f) - D^{\tilde{N}} _\cdot (g,f) \right \|_{L^\infty
     (\Z_h)} \le C \, h^r \] where $D^{\tilde{N}} _\cdot$ is the DVM
   operator defined in~\eqref{eq:DR} (for the precise quadrature
   weights derived in~\cite{HePa:DVM:02}) on the grid above-mentioned,
   and $f_i = f ( i h)$.  Here $r = k/(k+3)$ and the constant $C$ is
   independent on $h$.
 \end{theorem}
\medskip
\begin{remark}~\rm
As can be seen from Theorem~\ref{theo:HePa}, the
periodized DVM presented in this subsection is expected to have a
quite poor accuracy. On the contrary the spectral method
\cite{PePa:96}, even in the fast version of \cite{MoPa:2006}, has
been proven to be spectrally accurate, i.e. of infinite order for
smooth solutions. Nevertheless this periodized DVM has some
interesting features compared to the spectral method: preservation of
sign, stability, and preservation of the conservation laws up to
aliasing issues.
\end{remark}


\section{Fast DVM's algorithms}\label{sec:fastDVM}

The fast algorithms developed for the spectral method in
\cite{MoPa:2006} can be in fact extended to the periodized DVM
method. The method that originates was triggered by
  the reading of the direct FFT approach proposed in
\cite{BoRj:HS:97, BoRj:maxw:98, BoRj:HS:99}.

\subsection{Principle of the method: a pseudo-spectral viewpoint}

We start from the periodized DVM in $\llbracket-N,N \rrbracket^d$ with representation~\eqref{eq:DR} and as in the continuous case
we set, for $k,l \in -\tilde{N} \le k,l \le \tilde{N}$,
 \begin{equation*}
 \tilde{B}(|k|,|l|) = 2^{d-1} \,
 B \left( \frac{|k|}{\sqrt{|k|^2+|l|^2}},\sqrt{|k|^2+|l|^2} \right) \,
 (|k|^2+|l|^2)^{-\frac{d-2}2}.
 \end{equation*}
With this notation
 \begin{equation*}
 \widetilde{\Gamma}_{k,l} = {\bf 1}(k \cdot l) \,
 \tilde{B}(|k|,|l|) \, w_{k,l} ,
 \end{equation*}
and thus the DVM becomes
 \begin{equation*}
 \partial_t f_i = \sum_{-\tilde{N} \le k,l \le \tilde{N}} {\bf 1}(k \cdot l) \, \tilde{B}(|k|,|l|) \, w_{k,l}
 \, \big[ f_{i+k} f_{i+l} - f_i f_{i+k+l} \big].
 \end{equation*}
Now we transform this set of ordinary differential equations into
a new one using the involution transformation of the discrete Fourier
transform on the vector $(f_i)_{-N \le i \le N}$. This
involution reads for $I \in \llbracket-N,N \rrbracket^d$
 \[ \tilde{f}_I = \frac{1}{2N+1} \, \sum_{i=-N} ^{N} f_i \, {\bf e}_{-I} (i), \hspace{0.8cm}
 f_i = \sum_{I=-N} ^{N} \tilde{f}_I \, e_I (i) \]
where ${\bf e}_K (k)$ denotes $e^{\frac{2 i \pi \, K \cdot k}{2N+1}}$,
and thus the set of differential equations becomes
 \begin{eqnarray*}
 \partial_t \tilde{f}_I  &=& \sum_{K,L=-N} ^{N} \left( \frac{1}{2N+1} \, \sum_{i=-N} ^{N} {\bf e}_{K+L-I} (i) \right)
 \, \\&&\left[ \sum_{-\tilde{N} \le k,l \le \tilde{N}} {\bf 1}(k \cdot l) \, \tilde{B}(|k|,|l|) \, w_{k,l}
 \, \left( {\bf e}_K (k) {\bf e}_L (l) - {\bf e}_L (k+l) \right) \right] \, \tilde{f}_K \, \tilde{f}_L
 \end{eqnarray*}
for $-N \le I \le N$.
We have the following identity
 \[ \frac{1}{2N+1} \, \sum_{i=-N} ^{N} {\bf e}_{K+L-I} (i) = {\bf 1}(K+L-I) \]
and so the set of equations is
 \begin{equation} \label{eq:sysDiff}
 \partial_t \tilde{f}_I = \sum_{\substack{K,L=-N\\K+L=I}} ^{N}
 \tilde{\beta}(K,L) \, \tilde{f}_K \, \tilde{f}_L
 \end{equation}
with
 \begin{equation*}
 \tilde{\beta}(K,L) = \sum_{-\tilde{N} \le k,l \le \tilde{N}} {\bf 1}(k \cdot l) \,
 \tilde{B}(|k|,|l|) \, w_{k,l} \, \big[ {\bf e}_K (k) {\bf e}_L (l) - {\bf e}_L (k+l) \big] = \beta(K,L) - \beta(L,L)
 \end{equation*}
where
 \begin{equation} \label{eq:betaKL}
 \beta (K,L) = \sum_{-\tilde{N} \le k,l \le \tilde{N}} {\bf 1}(k \cdot l) \,
 \tilde{B}(|k|,|l|) \, w_{k,l} \, {\bf e}_K (k) {\bf e}_L (l).
 \end{equation}
 Let us first remark that this new formulation allows to reduce the
 usual cost of computation of a DVM exactly to $O(N^{2d})$ (as with
 the usual spectral method) instead of $O(N^{2d+\delta})$ for $\delta
 \sim 1$~\cite{Bu:96,MiSc:2000,HePa:DVM:02}. Note however that the
$(2N+1)^d \times (2N+1)^d$ matrix of coefficients $(\beta
(K,L))_{K,L}$ has to be computed and stored first, thus the storage
requirements are larger with respect to usual DVM.  Nevertheless
symmetries in the matrix can substantially reduce this cost.

Now the aim is to give an expansion of $\beta (K,L)$ of the form
 \begin{equation*}
 \beta_{K,L} \simeq \sum_{p=1} ^M \alpha_p (K) \, \alpha' _p (L),
 \end{equation*}
 for a parameter $M \in \NN^*$ to be defined later.
 Indeed, this formulation will allow us to write \eqref{eq:sysDiff} as
 a sum of $M$ discrete convolutions and then this algorithm
 can be computed in $O(M \, N^d \log_2(N))$ operations by using
 standard {FFT} techniques~\cite{CoTu:65,CHQ:88}, as in the fast
 spectral method.

\subsection{Expansion of the discrete kernel modes}
We make a decoupling assumption on the collision
kernel as in the spectral case~\cite{MoPa:2006}
 \begin{equation}\label{eq:decoup}
 \tilde{B}(|k|,|l|) \, w_{k,l} = a(k) \, b(l).
 \end{equation}

Note that the DVM constructed by quadrature in dimension $3$ for
hard spheres in~\cite{HePa:DVM:02} on the cartesian velocity grid 
$h \, \Z^3$ (for $h > 0$) satisfies this decoupling
assumption with $a(k) = h^5 \, |k| / \mbox{gcd}(k_1,k_2,k_3)$ and
$b(l)=1$ (see~\cite[Formula~(20-21)]{HePa:DVM:02}), and
$\mbox{gcd}(k_1,k_2,k_3)$ denotes the greater common divisor of
the three integers. For Maxwell molecules in dimension $2$ 
on the grid $h \, \Z^2$, these coefficients are 
$a(k) =  h^3 \,|k| / \mbox{gcd}(k_1,k_2)$ and $b(l) = 1$.

The difference here with the spectral method, which is a continuous
numerical method, is that we have to {\em enumerate} the set of
$\{-\tilde{N} \le k,l \le \tilde{N} \ | \ k \, \bot \, l \, \}$.  This
motivates for a detailed study of the number of lines passing through
$0$ and another point in the grid (this is equivalent to the study of
this set), in order to compute the complexity of the method in term of
$N$.

To this purpose let us introduce the Farey series and a new parameter
$0 \le \bar{N} \le \tilde{N}$ for the size of the grid used to compute
the number of directions.  The usual Farey series is
 \begin{equation*}
 \mathcal{F}_{\bar{N}} ^1 =
 \left\{ (p,q) \in \llbracket 0,\bar{N} \rrbracket^2 \ | \ 0 \le p \le
   q \le \bar{N}, \ q \ge 1, \ \mbox{and} \ \mbox{gcd}(p,q) = 1 \right\}
 \end{equation*}
 where $\mbox{gcd}(p,q)$ denotes again the greater common divisor of
 the two integers (more details can be found in~\cite{HardyWright}). 
  We gave a schematic representation of the two dimensional Farey series 
 in Figure \ref{fig:farey}. It is straightforward to see that the number
 of lines $A_{\bar{N}} ^1$ passing through $0$ in the grid $\llbracket
 -\bar{N},\bar{N} \rrbracket^2$ is
   \[ A_{\bar{N}} ^1 = 4 \left (\left |\mathcal{F}_{\bar{N}} ^1\right | - 1\right ),\]
 where the factor  $4$ allows to take into account the 
 permutations  when counting the couples $(p,q)$ as well as the ordering 
 (symmetries in Figure 1), minus the line which is repeated during
 the symmetry process.

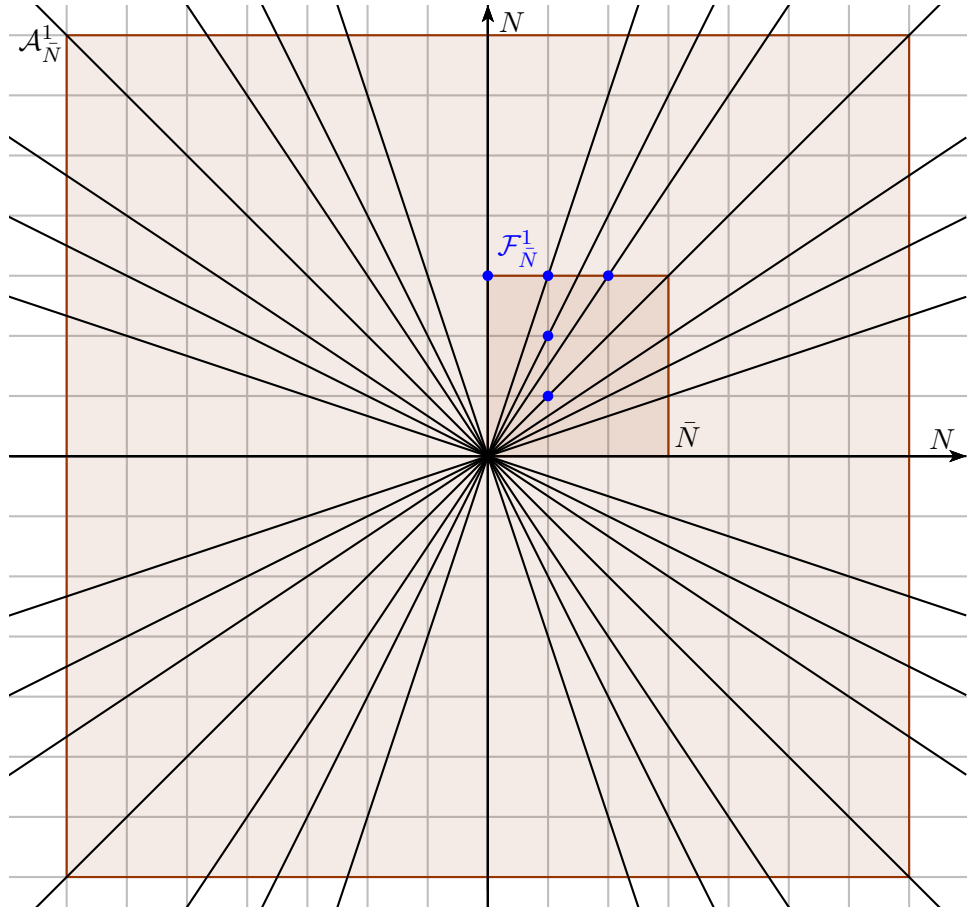
\begin{figure}
  \begin{center}
		\newrgbcolor{zzttqq}{0.6 0.2 0}
		\psset{xunit=0.8cm,yunit=0.8cm}
		\begin{pspicture*}(-7.95,-7.5)(7.95,7.5)
			\psgrid[subgriddiv=0,gridlabels=0,gridcolor=lightgray](0,0)(-7.95,-7.5)(7.95,7.5)
			\psset{xunit=0.8cm,yunit=0.8cm,algebraic=true,dotsize=3pt 0,linewidth=0.8pt,arrowsize=3pt 2,arrowinset=0.25}
			\psaxes[labelFontSize=\scriptstyle,xAxis=true,yAxis=true,labels=none,ticksize=0pt]{->}(0,0)(-7.95,-7.5)(7.95,7.5)[$N$,140] [$N$,-40]
			\pspolygon[linestyle=none,fillstyle=solid,fillcolor=zzttqq,opacity=0.1](7,7)(7,-7)(-7,-7)(-7,7)
			\pspolygon[linestyle=none,fillstyle=solid,fillcolor=zzttqq,opacity=0.1](0,3)(0,0)(3,0)(3,3)
			\psline[linecolor=zzttqq](7,-7)(-7,-7)
			\psline[linecolor=zzttqq](-7,-7)(-7,7)
			\psline[linecolor=zzttqq](-7,7)(7,7)
			\psline[linecolor=zzttqq](7,-7)(7,7)
			\psline[linecolor=zzttqq](0,3)(0,0)
			\psline[linecolor=zzttqq](0,0)(3,0)
			\psline[linecolor=zzttqq](3,0)(3,3)
			\psline[linecolor=zzttqq](3,3)(0,3)
			\psplot{-8.05}{7.95}{(-0-1*x)/-1}
			\psplot{-8.05}{7.95}{(-0-1*x)/-2}
			\psplot{-8.05}{7.95}{(-0-2*x)/-1}
			\psplot{-8.05}{7.95}{(-0-3*x)/-2}
			\psplot{-8.05}{7.95}{(-0-3*x)/-1}
			\psplot{-8.05}{7.95}{(-0-2*x)/-3}
			\psplot{-8.05}{7.95}{(-0-1*x)/-3}
			\psline(0,-7.23)(0,7.56)
			\psplot{-8.05}{7.95}{(-0-0*x)/-3}
			\psplot{-8.05}{7.95}{(-0-3*x)/1}
			\psplot{-8.05}{7.95}{(-0-2*x)/1}
			\psplot{-8.05}{7.95}{(-0-3*x)/2}
			\psplot{-8.05}{7.95}{(-0-1*x)/1}
			\psplot{-8.05}{7.95}{(-0-2*x)/3}
			\psplot{-8.05}{7.95}{(-0-1*x)/2}
			\psplot{-8.05}{7.95}{(-0-1*x)/3}
			\psdots[dotstyle=*,linecolor=blue,dotsize=4pt](1,1)
			\psdots[dotstyle=*,linecolor=blue,dotsize=4pt](1,2)
			\psdots[dotstyle=*,linecolor=blue,dotsize=4pt](1,3)
			\psdots[dotstyle=*,linecolor=blue,dotsize=4pt](2,3)
			\psdots[dotstyle=*,linecolor=blue,dotsize=4pt](0,3)			
			\rput[bl](-7.8,6.56){$\mathcal A_{\bar N}^1$}
			\rput[bl](0.16,3.15){\textcolor{blue}{$\mathcal F_{\bar N}^1$}}
			\psdots[dotstyle=*](0,0)
			\rput[bl](3.1,0.15){$\bar N$}
		\end{pspicture*}
	  \caption{Representation of the Farey series $\mathcal F_{\bar N}^1$ and of  $\mathcal A_{\bar N}^1$, the primal representant of lines in $[-N,N]$ associated, for $N=7$ and $\bar N = 3$}
	  \label{fig:farey}
  \end{center}
\end{figure}
 
 Similarly one can define the set
 \begin{equation*}
   \mathcal{F}^2 _{\bar{N}} = \left\{ (p,q,r) \in \llbracket 0,\bar{N} \rrbracket^3 \ |
     \ 0 \le p \le q \le  r \le \bar{N}, \ r \ge 1, \ \mbox{and} \ \mbox{gcd}(p,q,r) = 1 \right\}
 \end{equation*}
 and the number of lines $A^2 _{\bar{N}}$ passing through $0$ in the
 grid $\llbracket -\bar{N},\bar{N} \rrbracket^3$ is 
 \[A^2_{\bar{N}} = 24 \left ( \left |\mathcal{F}^2_{\bar{N}}\right |
                   - \left |\mathcal{F}^1 _{\bar{N}}\right |\right )    
                   - 2 \, A^1_{\bar N}
   \] 
 all possible permutations of the three numbers times $4$ and minus the interfaces 
 $2 A^1_{\bar N}$ accounting  for the possible negative values by symmetry, minus 
 $24 \left |\mathcal{F}^1 _{\bar{N}}\right |$ for the spurious terms
 when two equal numbers are swapped.
 The exponents of the Farey series refer to the dimension of the 
 space of lines (which is $d-1$).  Now let us estimate the cardinals 
 of $\mathcal{F}^1 _{\bar{N}}$ and $\mathcal{F}^2 _{\bar{N}}$.
 
 \begin{lemma}\label{lem:farey}
   The Farey series in dimension $d=2$ and $d=3$ satisfy the following
   asymptotic behavior
	 \begin{eqnarray*}
	 \left| \mathcal{F}^1 _{\bar{N}} \right | &=& \frac{\bar{N}^2}{2 \, \zeta(2)} + O(\bar{N} \, \log \bar{N})
	 = \frac{3 \bar{N}^2}{\pi^2} + O(\bar{N} \, \log \bar{N}), \\ 
	 \left |\mathcal{F}^2 _{\bar{N}}\right | &=& \frac{\bar{N}^3}{{12} \, \zeta(3)} + O(\bar{N}^2),
	 \end{eqnarray*}
	 where $\zeta(s) = \sum_{n \ge 0} n^{-s}$ denotes the usual Riemann
	 zeta function.
 \end{lemma}
\medskip
\begin{remark}~\rm
In dimension $d$, the formula would be
 \begin{multline*}
   \mathcal{F}^{d-1} _{\bar{N}} = \Big\{ (p_1,p_2,\dots,p_d) \in \llbracket 0,\bar{N} \rrbracket^d \ | \,
   0 \le p_1 \le p_2 \le \dots \le p_d \le \bar{N},  \ p_d \geq 1 \\ \mbox{and gcd}(p_1,p_2,\dots,p_d) = 1 \Big\}.
 \end{multline*}
The cardinal of $\mathcal{F}^{d-1} _{\bar{N}}$
could be computed by induction with the same tools as in the proof:

 \[ \left |\mathcal{F}^{d-1} _{\bar{N}}\right | = C_d \, \frac{\bar{N}^d}{\zeta(d)} + O(\bar{N}^{d-1}). \]
The non-negative constant $C_d$ is given by
\[C_d :=  \frac{1}{2^{d-2} \, d!} , \]
the factorial coming from the successive summations of the Riemann series.

\end{remark}
\begin{proof}[Proof of Lemma~\ref{lem:farey}]
  The proof of the first equality is extracted from~\cite[Theorems~330
  \& 331 page 268]{HardyWright}, and given shortly for convenience of
  the reader. The proof of the second inequality is inspired from this
  first proof.

  Let us introduce $\varphi(n)$ the Euler function (i.e. the number of
  integers less than and prime to $n$) and the multiplicative M\"obius
  function $\mu (n)$ such that $\mu (1) =1$, $\mu (n) = 0$ if $n$ has
  a squared factor and $\mu (p_1 p_2 \cdots p_k) = (-1)^k$ if all the
  primes $p_1, p_2, \dots, p_k$ are different. We have the following
  connection between these two arithmetical functions (see
  \cite[Formula~(16.3.1), page 235]{HardyWright}):
 \begin{equation*}
 \varphi(n) = n \sum_{d |n} \frac{\mu(d)}{d} = \sum_{d d' = n} d' \mu(d).
 \end{equation*}

Now let us compute the cardinal of the Farey series in dimension $2$:
 \begin{align*}
 \left |\mathcal{F}_{\bar{N}} ^1\right | = & \ \varphi(1) + \cdots + \varphi(\bar{N})
                      = \sum_{m=1} ^{\bar{N}} \sum_{dd' = m} d' \mu(d) \\
                    =& \ \sum_{dd' \le \bar N} d' \mu(d) = 1 +
                    \sum_{d=1} ^{\bar{N}} \mu (d) \left( \sum_{d' =1}
                      ^{[\bar{N}/d]} d' \right) \\
                    =& \ \frac12  \sum_{d=1} ^{\bar{N}} \mu (d) \left( [ \bar{N}/d ]^2 + [ \bar{N}/d ] \right)
                        = \frac12  \sum_{d=1} ^{\bar{N}} \mu (d) \left( (\bar{N}/d)^2 + O(\bar{N}/d) \right) \\
                    =& \ \frac{\bar{N} ^2}2  \sum_{d=1} ^{\bar{N}} \frac{\mu (d)}{d^2} +
                        O \left( \bar{N} \sum_{d=1} ^{\bar{N}} \frac1d \right)
                        = \frac{ \bar{N} ^2}2  \sum_{d=1} ^{\infty} \frac{\mu (d)}{d^2}
                           + O \left( \bar{N} ^2 \sum_{\bar{N}+1} ^\infty \frac{1}{d^2} \right) \\
                       & + \ O \left( \bar{N} \log \bar{N} \right)
                    =  \frac{\bar{N}^2}{2 \, \zeta(2)} + O (\bar{N}) + O \left( \bar{N} \log \bar{N} \right)
 \end{align*}
where we have used the classical formula $1/\zeta(s) = \sum_{n=1}
^\infty \mu(n)/n^s$ (cf. \cite[Theorem~287, page 250]{HardyWright}).

Now for the dimension $d=3$, we enumerate the set
$\mathcal{F}_{\bar{N}} ^2$ in the following way: we fix $r \ge 1$ then
$1 \le q \le r$ (the case $q=0$ is trivial and treated separately),
then $p \le q$ such that $\mbox{gcd}(p,\mbox{gcd}(q,r))=1$ (we use the
associativity of the function $\mbox{gcd}$). This leads us to count
the number of $p$ in $\llbracket 1,q\rrbracket$ such that
$\mbox{gcd}(p,\delta)=1$ for a given $\delta | q$.  When $\delta >1$,
writing $p = k \delta + p_0$ with $p_0 \in \llbracket
1,\delta-1\rrbracket$, this number is seen to be $\varphi(\delta) \,
(q/\delta)$. When $\delta =1$ this number is $q+1$ (all the values
from $0$ to $q$). Thus the formula $\varphi(\delta) \, (q/\delta)$ is
still valid if we deal separately with the case $p=0$, which has
cardinal $\left |\mathcal{F}_{\bar{N}} ^1\right |$.  Now let us
compute the cardinal of $\mathcal{F}_{\bar{N}} ^2$. We first write 
 \begin{align}
 \left |\mathcal{F}_{\bar{N}} ^2\right | &= \left |\mathcal{F}_{\bar{N}} ^1\right |
                        + \sum_{r=1} ^{\bar{N}} \sum_{q =1} ^r q \, \frac{\varphi\left (\mbox{gcd}(q,r)\right )}{\mbox{gcd}(q,r)} \notag \\
                   & = \frac{\bar{N}^2}{2 \, \zeta(2)} + O\left( \bar{N} \log \bar{N} \right) 
                       + \sum_{r=1} ^{\bar{N}} \sum_{q =1} ^r q \, \sum_{d |q, \, d|r} \frac{\mu(d)}d \notag \\
                   & = O\left( \bar{N}^2 \right) + \frac12 \, \sum_{d=1} ^{\bar{N}} \frac{\mu(d)}d
                       \sum_{\substack{r = 1 \\ d|r}}^{\bar N} \sum _{\substack{q = 1 \\ d|q}}^r q. \label{eq:cardFN2}
 \end{align}
 We shall now focus on the last member of the right hand side of this expression. We have
 \begin{align}
   \sum_{\substack{r = 1 \\ d|r}}^{\bar N} \sum _{\substack{q = 1 \\ d|q}}^r q
                   & = d \sum_{\substack{r = 1 \\ d|r}}^{\bar N} \sum_{d' = 1}^{[r/d]} d'
                     = \frac{d}{2} \sum_{\substack{r = 1 \\ d|r}}^{\bar N}\left ( \left [\frac{r}{d}\right ]^2 + \left [\frac{r}{d}\right ] \right ) \notag \\
                   & = \frac{d}{2} \sum_{d'' = 1}^{[\bar N/d]}\left ( (d'')^2 + d'' \right ) \notag \\
                   & = \frac d2 \left ( \frac 13 \left ( \bar N/d\right )^3 + O \left ( (\bar N/d)^2\right ) +  O \left ( \bar N/d\right ) \right ). \label{eq:sumRQ}
 \end{align}
Finally, we obtain by plugin \eqref{eq:sumRQ} into \eqref{eq:cardFN2}
 \begin{align*}
 \left |\mathcal{F}_{\bar{N}} ^2\right | & =  O\left( \bar{N}^2 \right)
                        + \frac14 \, \sum_{d=1} ^{\bar{N}} \mu(d) \, \left ( \frac 13 \left (\bar N/d\right )^3 
                        + O \left ( (\bar N/d)^2\right ) +  O \left ( \bar N/d\right ) \right ) \\
                   & = O\left( \bar{N}^2 \right) + \frac{\bar{N}^3} {12} \, \sum_{d=1} ^{\bar{N}} \frac{\mu(d)}{d^3} +
                        O \left( \bar{N}^2 \, \sum_{d=1} ^{\bar{N}} \frac{\mu(d)}{d^2} \right)
                        + O \left( \bar{N} \, \sum_{d=1} ^{\bar{N}} \frac{\mu(d)}{d} \right) \\
                   & = \frac{\bar{N}^3}{12} \, \sum_{d=1} ^{+\infty} \frac{\mu(d)}{d^3} +
                       O\left(\bar{N}^3 \, \sum_{d=\bar{N}+1} ^{+\infty} \frac{1}{d^3}
                       \right) + O \left ( \bar{N}^2\right ) \\
                   & = \frac{\bar{N}^3}{12 \, \zeta(3)} + O\left (\bar{N}^2\right ).
 \end{align*}
 This conclude the proof.
\end{proof}

Now one can deduce the following decomposition of the kernel modes using their
definition \eqref{eq:betaKL} and the decoupling assumption \eqref{eq:decoup} on the discrete kernel
 \begin{eqnarray*}
   \beta (K,L) &=& \sum_{-\tilde{N} \le k,l \le \tilde{N}} {\bf 1}(k \cdot l) \,
   a(|k|) \, b(|l|) \,  e_K (k) e_L (l) \\
   &\simeq& \beta^{\bar N} (K,L) =
   \sum_{e \in \mathcal{A}_{\bar{N}}^{d-1}} \Big [ \sum_{\substack{k \in e \Z \\ -\tilde{N} \le k \le \tilde{N}}}
   a(|k|) \, e_K (k) \Big ] \, \Big [ \sum_{\substack{l \in e^\bot \\ -\tilde{N} \le l \le \tilde{N}}}
   b(|l|) \, e_L (l) \Big ]
 \end{eqnarray*}
with equality if $\bar{N}= \tilde{N}$. Here
$\mathcal{A}_{\bar{N}}^{d-1}$ denotes the set of primal
representants of directions of lines in
$\llbracket-\bar{N},\bar{N} \rrbracket$ passing through $0$. After
indexing this set, which has cardinal $A^{d-1}_{\bar{N}}$, one
gets
 \begin{equation}\label{eq:decDVM}
   \beta^{\bar N}(K,L) = \sum_{p=1}
   ^{A^{d-1}_{\bar{N}}} \alpha_p (K) \, \alpha' _p (L)
 \end{equation}
with
 \[ \alpha_p (K) = \sum_{\substack{k \in e_p \, \Z \\ -\tilde{N} \le k \le \tilde{N}}}  a(|k|) \, e_K (k), \hspace{0.8cm}
 \alpha' _p (L) = \sum_{\substack{l \in e_p ^\bot \\ -\tilde{N} \le l
     \le \tilde{N}}} b(|l|) \, e_L (l). \]
After inversion of the discrete Fourier transform, this method yields a
decomposition of the discrete collision operator
\begin{equation}\label{eq:decD}
 D_i^{\tilde{N}} \simeq D_i^{\tilde{N},\bar{N}} =
\sum_{p=1} ^{A^{d-1} _{\bar{N}}}  D_i^{\tilde{N},\bar{N},p},
\qquad i \in \llbracket -N,N\rrbracket^d,
\end{equation}
with equality with \eqref{eq:DR} if $\bar{N}= \tilde{N}$. Each
$D_i^{\tilde{N},\bar{N},p}(f,f)$ is defined by the $p$-th term of the
decomposition of the kernel modes~\eqref{eq:decDVM}. Each term
$D^{\tilde{N},\bar{N},p}$ of the sum is a discrete convolution
operator when it is written in Fourier space.
Moreover, each $\alpha_p$ (resp. $\alpha'_p$) is 
defined as the discrete Fourier transform of some non-negative 
coefficients $a(|k|)$ times 
the characteristic function of $k \in e_p \Z$ (resp. $b(|l|)$ 
times the characteristic function of $l \in e_p^\perp$). Hence, we get
after inversion of the transform that  $D^{\tilde N, \bar N, p}$
is a discrete convolution with non-negative  coefficients.
 
By using the approximate kernel modes $\beta^{\bar N}(K,L)$, we
obtain a new discrete evolution equation, which heritates the same
nice stability properties as the usual DVM schemes, as stated in the
following proposition. Its proof is exactly similar to the one of
Proposition~\ref{prop:stab}, when computing by inverse Fourier
transform the coefficients $\widetilde \Gamma^{\bar N}_{k,l}$ associated
to the approximate kernel modes $\beta^{\bar N}(K,L)$.

\begin{proposition}
  \label{prop:stab:fast}
  Assume that the quadrature weight $w_{i,j}^{k,l}>0$ are
  positive. Consider some truncation numbers $\bar N \le \tilde N \le
  N \in \N^*$ and some non-negative initial data $f_i(0) \ge 0$, $i
  \in \llbracket -N,N\rrbracket^d$. Then the discrete evolution
  equation
\begin{equation} \label{eq:fastDVM}
 \partial_t f_i = D_i^{\tilde{N},\bar{N}}(f,f),
\qquad i \in \llbracket -N,N\rrbracket^d,
 \end{equation}
 is globally well-posed in $\R^{\llbracket
   -N,N\rrbracket^d}$. Moreover the coefficients $f_i(t)$ are
 non-negative for all time, and
\[
\forall \, t \ge 0, \quad \sum_{i \in \llbracket -N,N\rrbracket^d}
f_i(t) = \sum_{i \in \llbracket -N,N\rrbracket^d} f_i(0).
\]
\end{proposition}

\begin{remark}
	Using the non-negativity of the coefficients together with
	the conservation of mass, momentum and energy, we can prove
	thanks to standard arguments (see~\cite{CaGa:80}) that the
	discrete entropy of solutions to the fast DVM method is
	non-increasing in time.
\end{remark}


\subsection{Implementation of the algorithm} \label{sub:implement}

The fast DVM method described in the last subsection depends on the
three parameters $N$ (the size of the gridbox), $R$ (the truncation
parameter) and $\bar N$ (the size of the box in the space of
lines). The only constraint on these parameters is the no-aliasing
condition that relates $R$ and the size of the box (and thus $R$ and
$N$, thanks to the parameter $\tilde{N}$). 

 Thus one can see thanks to Lemma \ref{lem:farey} that even if we take
 $\bar{N}=\tilde{N}=N$, i.e. we take all possible directions in the
 grid $\llbracket -N,N\rrbracket^d$, we get the computational cost
 $O(N^{2d} \log_2 N)$ which is better than the usual cost of the DVM,
 $O(N^{2d+1})$ (but slightly worse than the cost $O(N^{2d})$ obtained
 by solving directly the pseudo-spectral scheme, thanks to a bigger
 storage requirement).

 More generally for a choice of $\bar{N} < N$ we obtain the cost
 $O(\bar{N} ^d N^d \log_2 N)$, which is slightly
 worse than the cost of the fast spectral algorithm (namely $O(M^{d-1}
 N^d \log_2 N)$ where $M$ is the number of discrete
 angle~\cite{MoPa:2006}), but interesting given that the algorithm is
 accurate for small values of $\bar N$, and more
   stable. The justification for this is the low accuracy of the
 method (the reduction of the number of direction has a small effect
 on the overall accuracy of the scheme).

Finally, as for the fast spectral algorithm, the decomposition
\eqref{eq:decD} is completely parallelizable and the computational
cost should be reduced (formally) on a parallel machine up to $O(N^d
\log_2 N)$. This method also has the same adaptivity of the fast
spectral algorithm: in a space inhomogeneous setting, the parameter
$\bar N$ can be made space dependent, according to the fact that some
regions in space deserve less accuracy than others, being close to
equilibrium.

\begin{remark}~
\begin{enumerate}
\item Concerning the construction of the set of directions
$\mathcal{A}^d _{\bar{N}}$, it can be done with systematic
algorithms of iterated subdivisions of a simplex, thanks to the
properties of the Farey series. In dimension $d=2$ this
construction is quite simple (see~\cite{HardyWright}). In
dimension $3$ we refer to~\cite{NoSe:03}.

\item Let us remark that in order to get a {\em regular} scheme
(i.e with no other conservation laws than the usual ones) in spite
of the reduction of directions, it is enough that the schemes
contains the directions $0$ and $\pi/2$ (see~\cite{Cerc:75}). This
is satisfied if we take the directions contained in $\mathcal{F}_1
^{d-1}$, i.e. as soon as $\bar{N} \ge 1$.

\item Finally in the practical implementation of the algorithm one has
  to take advantage of the symmetry of the
  decomposition~\eqref{eq:decDVM} in order to reduce the number of
  terms in the sum: for instance in dimension $2$, if $a=b=1$, one can
  write a decomposition with $A^{d-1} _{\bar{N}} /2$ terms.
\end{enumerate}
\end{remark}

\section{Numerical Results}
\label{sec:Numerics}

We will present in this Section some numerical results for the space
homogeneous Boltzmann equation in dimension $2$, with Maxwell
molecules. We will compare the fast DVM method presented in Section
\ref{sec:fastDVM} with the method introduced in~\cite{HePa:DVM:02}
(this latter method shall be referred to as the \emph{classical DVM}
one).  
The time discretization is
performed by a total variation diminishing second order Runge-Kutta method.

The first remark concerning the numerical simulations is that, thanks
to the discrete velocity approach, the conservations of mass, momentum
and energy is only affected by the aliasing error and thus, for a
sufficiently large computational domain, it is exact up to machine
precision. This is a relevant advantage compared to the spectral
(classical of fast) methods, where only mass (and momentum if one
considers symmetric distributions) is conserved exactly.


Let us now present some accuracy tests. In the case of two dimensional
Maxwell molecules, we have an exact solution of the homogeneous
Boltzmann equation given by
\begin{equation*}
f(t,v) = \frac{\exp(-v^2/2S)}{2\pi\,S^2} \,\left[2\,S-1+\frac{1-S}{2 \,S}\,v^2 \right]
\end{equation*}
with $S = S(t) = 1-\exp(-t/8)/2$. It corresponds to the well known
``BKW'' solution, obtained independently in~\cite{Bobylev:75} and
~\cite{KrookWu:1977}. This test is performed to check the accuracy of
the method, by comparing the error at a given time $T_{end}$ when
using $N=8$ to $N=128$ grid points for each coordinate (the case
$N=128$ for the classical DVM has been omitted due to its large
computational cost). We give the results obtained by the classical DVM
method and the fast one, with different numbers of $\bar N$.
We choose the value $\tilde N$ such that the classical method is convergent
according to Theorem \ref{theo:HePa}, namely
\begin{equation*}
    \tilde{N} = \left [\frac{2N}{3 + \sqrt{2}} \right ].
\end{equation*}
Then, one has $\tilde N = 1$ when $N=8$, $\tilde N = 3$ when
$N=16$, $\tilde{N} = 7$ when $N=32$ and $\tilde{N} = 14$ when $N =
64$. These values give a result corresponding to the kernel mode
$\eqref{eq:betaKL}$, namely that no truncation of the number of
lines has been done: the solution obtained is essentially the same
obtained with the classical DVM method.
 Note that $\bar N$ must be chosen less of equal than $\tilde{N}$ and this
is why we do not present the results with, e.g., $N=16$ and $\bar
N=7$.


\begin{table}

\begin{tabular}{|c|c|c|c|c|c|}
\hline
Number of  & Classical & Fast DVM        & Fast DVM        & Fast DVM        & Fast DVM\\
points $N$ & DVM       & with $\bar N=1$ & with $\bar N=3$ & with $\bar N=7$ & with $\bar N=14$ \\
\hline
  8        & 1.445E-3  & 1.4511E-3 &     x     &     x     &     x     \\
\hline
  16       & 8.912E-4 & 9.887E-4  & 8.9646E-4 &     x     &     x     \\
\hline
  32       & 6.1054E-4  & 6.5209E-4 & 5.8397E-4 & 6.1328E-4 &     x     \\ 
\hline
  64       & 2.6351E-4  & 4.094E-4  & 2.906E-4  & 3.667E-4  & 2.7341E-4 \\ 
\hline
  128      &     x     & 2.6669E-4 & 1.8245E-4 & 2.0371E-4 & 1.6341E-4 \\ 
\hline
\end{tabular}
\medskip
\caption{Comparison of the $L^1$ error between the classical DVM
method and the fast DVM method with different values of $\bar N$ at
time $T=0.01$, after one iteration.}
\label{tab:compError1}
\end{table}

Table \ref{tab:compError1} shows the relative $L^1$ error between the
exact ``BKW'' solution $f$ and the approximate one $f_i$. It is
defined by
\[ \mathcal{E}_1(t) = \frac{\sum_{i=-N}^{N}|f_i(t)-f(v_i,t)|}
{\sum_{i=-N}^{N} |f_i(t)|}.
\]
The size of the domain has to be chosen carefully in order to
minimize the aliasing error. In this test, we used $T=5$ for
$N=8$, $T=5.5$ for $N=16$, $T=7$ for $N=32$ and $T=8$ for $N=64, \
128$.

We can see that, even with very few directions, there is a small
loss of accuracy for the fast DVM method compared to the classical
one, and that taking all possible directions we recover the
original DVM solution. The observed order of convergence in $N$ is
close to $1$, as predicted by Theorem \ref{theo:HePa} and nearly
the same for all values of the truncation parameter $\bar N$ (with
a small loss for $\bar N=1$).

We also observe that the method is convergent with respect
to $\bar N$, although being not necessarily monotone (in the sense that the accuracy can be better for a fixed
couple of parameters $(N,\bar N_1)$, $\bar N_1 < N$, compared to the result 
obtained with another couple $(N,\bar N_2)$ with $\bar N_1 < \bar N_2 < N$). 
This is due to the very irregular discrete sphere associated with the Farey series,
which implies that the information contained in the kernel modes can be
more complete with the Farey series $\mathcal F_{\bar N_1}^1$ rather than
$\mathcal F_{\bar N_2}^1$.

 \begin{figure}
     \begin{center}
     \includegraphics[scale=1.]{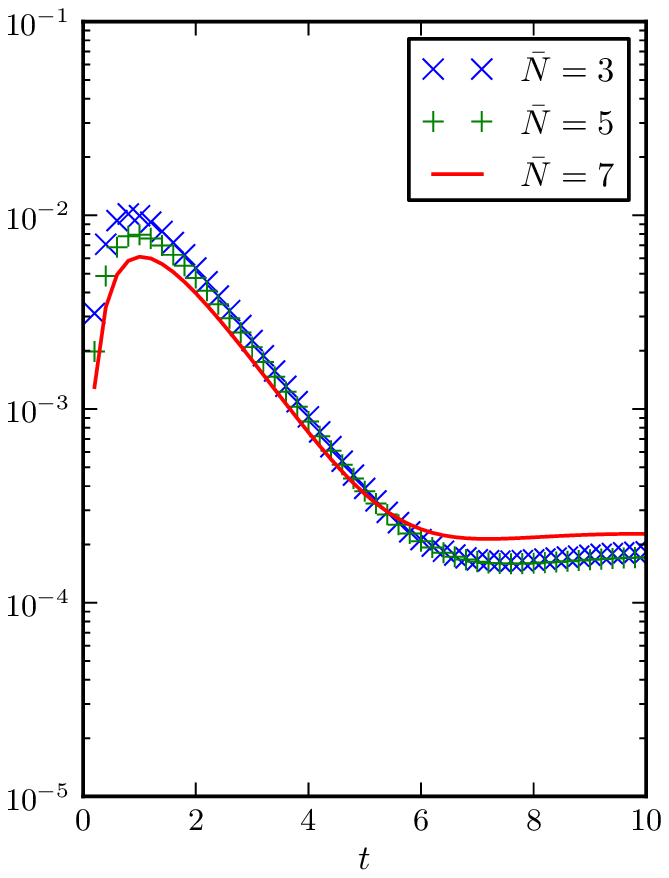}
     \includegraphics[scale=1.]{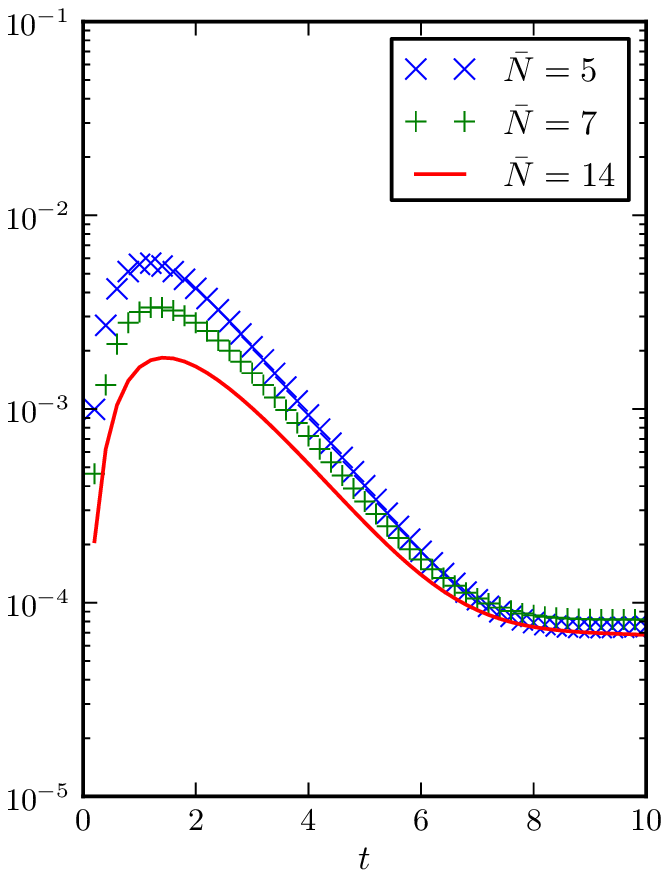}
     \caption{Evolution of the numerical relative $L^1$ error 
     of the fast DVM method for $N=32$ (left) and $N=64$ (right)
      grid points.}
     \label{fig:Evol}
     \end{center}
 \end{figure}

We then compare in Figure~\ref{fig:Evol} the time evolution of
this error, still in $L^1$ norm. We can see that it increases
initially (exactly as in the classical and fast spectral
methods~\cite{FiMoPa:2006}), and then decreases monotonically in
time. A saturation phenomenon due to aliasing errors finally
occurs as for the fast spectral method (see \cite{FiMoPa:2006}, 
Figure 1).

We then give the computational cost of the classical and fast DVM
methods in Table \ref{tab:compTime}. Here one can see the drastic
improvement when comparing the two methods: taking \emph{e.g.}
$N=64$ points in each direction, the fast method is more than $28$
time faster than the classical one when no truncation is done
(i.e. when we take $\bar N = \tilde N = 14$), and even $109$ times
faster with a small loss of accuracy when taking $\bar N = 7$.

We also present the evolution of these computational times with
respect to the total number of points in Figure \ref{fig:Time}. It
is clear when we look at the interpolant curve that the
theoretical predictions and the effective computational costs
agree perfectly. When $\bar N$ is fixed, the fast DVM method is of
order $N^2 \log(N)$ whereas when $N$ is fixed, the dependence in
$\bar N$ is very close to $\bar N^2$ (actually, the slope of the
interpolant curve is about $1.9$).

\begin{table}

\begin{tabular}{|c|c|c|c|c|c|}
\hline
Number of & Classical & Fast DVM        & Fast DVM        & Fast DVM         & Fast DVM \\
points $N$& DVM       & with $\bar N=3$ & with $\bar N=7$ & with $\bar N=14$ & with $\bar N=28$ \\
\hline
$16$      & 2 s. 95        & 0 s. 5  &       x      &       x      &       x       \\
\hline
$32$      & 2 min. 18 s.   & 3 s. 19 & 14 s. 52     &       x      &       x       \\
\hline
$64$      & 133 min. 44 s. & 16 s. 2 & 73 s. 4      & 4 min. 43 s. &       x       \\
\hline
$128$     &        x       & 85 s. 8 & 6 min. 18 s. & 23 min. 2 s. & 92 min. 11 s. \\
\hline
\end{tabular}
\medskip
\caption{Comparison of the computational time between the
classical DVM method and the fast DVM method with different values
of $\bar N$ at time $T=1$ ($\Delta t = 0.01$).} \label{tab:compTime}
\end{table}

 \begin{figure}
     \begin{center}
     \includegraphics[scale=1.]{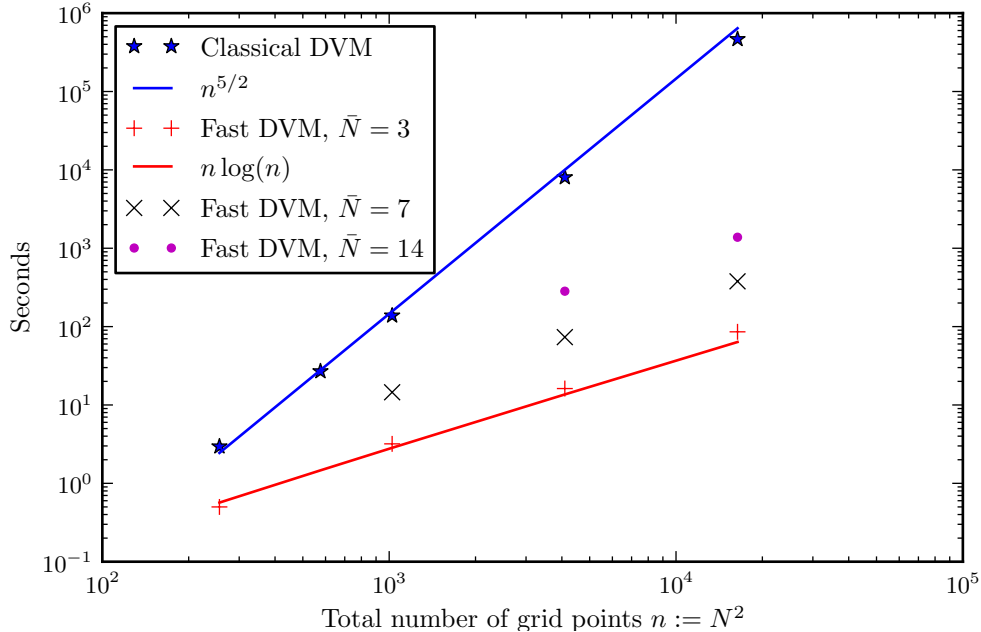}
     \caption{Evolution of the computational time with respect to the total number of points for the classical and fast DVM methods, at time $T=1$}
     \label{fig:Time}
     \end{center}
 \end{figure}

\section{Conclusions}
We have presented a deterministic way for computing the Boltzmann
collision operator with fast algorithms. The method is based on a
Carleman-like representation of the operator that allows to
express it as a combination of convolutions (this is trivially
true for the loss part but it is not trivial for the gain part). A
suitable periodized truncation of the operator is then used to
derive fast algorithms for computing discrete velocity models (DVM).
This can be adapted to any DVM, provided it features a
decoupling properties on the quadrature nodes. Our approach
will bring the overall cost in dimension $d$ to $O(\bar{N} ^d N^d \log_2 N)$
where $N$ is the size of the velocity grid and $\bar{N}$ is the size
of the grid used to compute directions in the approximation of the
discrete operator. Numerical evidences show that the quantity $\bar N$
can be taken small compared to $N$. Consistency and accuracy of the proposed
schemes are also presented, both theoretically and numerically.

\bigskip

\noindent {\bf{Acknowledgments.}} The first author wishes to thank
Bruno S\'evennec for fruitful discussions on the Farey series.
The third author wishes to thank Francis Filbet for fruitful discussions and comments
about the implementation of the numerical method.

\medskip

\bibliographystyle{acm}

\bibliography{bibMPRfast}

\begin{flushleft} \signcm \end{flushleft}
\vspace*{-45mm}
\begin{flushright} \signlp \end{flushright}
\begin{flushleft} \signtr \end{flushleft}

\end{document}